\documentclass[11pt]{amsproc}
\usepackage{amsfonts}
\usepackage{amsmath,amstext,amsbsy,amssymb}

\newtheorem{theorem}{Theorem}[section]

\newtheorem{proposition}[theorem]{Proposition}
\newtheorem{corollary}[theorem]{Corollary}

\theoremstyle{definition}
\newtheorem{definition}[theorem]{Definition}

\theoremstyle{remark}
\newtheorem{remark}[theorem]{Remark}

\numberwithin{equation}{section}

\newcommand{\Pf}{\mathrm{Pf}}

\begin{document}
\title[Deformation of Cayley's hyperdeterminants]{Deformation of Cayley's hyperdeterminants}
\author{Tommy Wuxing Cai}
\address{
Department of Mathematics, University of Manitoba, Winnipeg, MB R3T 2N2 Canada
}
\email{cait@myumanitoba.ca}
\author{Naihuan Jing}
\address{Department of Mathematics,
   North Carolina State University,
   Raleigh, NC 27695-8205, USA}
\email{jing@math.ncsu.edu}
\thanks{*Corresponding author: Naihuan Jing}

\begin{abstract}
We introduce a deformation of Cayley's second hyperdeterminant for even-dimensional hypermatrices.
As an application, we formulate a generalization of the Jacobi-Trudi formula for Macdonald functions of rectangular
shapes generalizing Matsumoto's formula for Jack functions.
\end{abstract}
\maketitle
\section{Introduction}

For a $2m$ dimensional hypermatrix $A=(A(i_1,\dotsm,i_{2m}))_{1\leq i_1,\dotsc,i_{2m}\leq n}$ Cayley's
second hyperdeterminant \cite{Ca} of $A$ is a generalization of the usual determinant:
\begin{align}\label{D:hyperdeterminant}
\mathrm{det}^{[2m]}(A)=\frac{1}{n!}\sum_{\sigma_1,\dotsc,\sigma_{2m}\in\mathfrak{S}_{n}}
\prod_{i=1}^{2m}\operatorname{sgn}(\sigma_i)\prod_{i=1}^n A(\sigma_1(i),\dotsc,\sigma_{2m}(i)).
\end{align}

In \cite{GKZ} properties of hyperdeterminants and discriminants are studied in general. In particular, it is
easy to see that $\mathrm{det}^{[2m]}(A)$ is invariant when $A$ is replaced by $B\circ_kA$, where the action
is defined for example,
$(B\circ_1A)(i_1,\dotsm,i_{2m})=\sum_{j=1}^n B_{i_1,j}A(j,i_1,\dotsm,i_{2m})$ for any $B\in \mathrm{GL}_n$.
This means that the hyperdeterminant is an invariant under the action of $\mathrm{GL}_n^{\otimes 2m}$.
In general, it is a challenging problem to come up with relative invariants under
the (non-diagonal) action of $\mathrm{GL}_n\times\cdots\times\mathrm{GL}_n$.

A $q$-analog Hankel determinant has been studied in \cite{ITZ} (see also \cite{BBL}) and $q$-analog of the hyperdeterminant for non-commuting matrices has been
introduced in the context of quantum groups \cite{JZ}.

In this paper, we introduce a $\lambda$-hyperdeterminant (of commuting entries) by
replacing permutations with alternating sign matrices in a natural and nontrivial manner. We will
show that the new hyperdeterminant deforms Cayley's hyperdeterminant.
As an application, we give a Jacobi-Trudi type formula for Macdonald polynomials to generalize
the previous formula of Matsumoto for Jack polynomials of rectangular shapes.
It would be interesting to uncover the relation with $q$-Hankel determinants. Also
the relationship with
$q$-deformations remains mysterious.


%
%

\section{$\lambda$-hyperdeterminants}

%

To introduce the $\lambda$-hyperdeterminant, we first consider generalization of the permutation sign in the usual
determinant, and then generalize Cayley's hyperdeterminants.

A permutation matrix is any square matrix
obtained by permuting the rows (or columns) of the identity matrix, therefore its determinant
is equal to the sign of the corresponding permutation. Each row or column of a fixed permutation matrix
has only one nonzero entry 
along each row and each column. Generalizing the notion of permutation matrices,
an alternating sign matrix is a square matrix with entries $0$'s, $1$'s and $(-1)$'s such that the sum of each row and column is $1$ and the nonzero entries in each row and column alternate in sign. Therefore, each permutation matrix is an alternating sign matrix. However, there
are more alternating sign matrices than permutation matrices of the same size. For example, the unique alternating sign $3\times3$ matrix which is not a permutation matrix is
\begin{equation}\label{e:3perm}
Q=\left[\begin{matrix}0&1&0\\1&-1&1\\0&1&0\end{matrix}\right].
\end{equation}

Clearly
the nonzero entries in each row or column of an alternating sign matrix
must start with $1$. We denote by $\mathrm{Alt}_n$ the set of $n$ by $n$ alternating sign matrices
and $P_n$ the subset of permutation matrices of size $n$. The set $P_n$, endowed with the usual matrix multiplication,
is a group isomorphic to the
symmetric group $\mathfrak S_n$. As $n$ increases, $\mathrm{Alt}_n$ contains more and more non-permutation matrices.
In fact, it is known \cite{Z, K} that $\mathrm{Alt}_n$ has a cardinality of $\prod_{i=0}^{n-1}\frac{(3i+1)!}{(n+i)!}$.

For $X=(x_{ij})\in \mathrm{Alt}_n$, we define the inversion number of $X$ by \cite{MRR}
 $$i(X)=\sum_{r>i,s<j}x_{ij}x_{rs}.$$
If $P=(\delta_{i,\pi_i})$ is the permutation matrix associated to the permutation $\pi\in\mathfrak S_n$,
Then $i(P)$ counts the number of pairs $(i, j)$ such that $i<j$ but $\pi_i>\pi_j$, i.e., $i(P)$ is
the inversion number of $\pi$, thus explains the name for alternating sign matrices.

 Let $n(X)$ be the number of negative entries in $X$.
 For any $n\times n$ matrix $A=(a_{ij})$, we write $A^X=\prod_{i,j}a_{ij}^{x_{ij}}$. As usual if $x_{ij}=0$, $a_{ij}^{x_{ij}}=1$.
 In general $A^X$ is a quotient of two monomials and the degree of the denominator being $n(X)$.

Let $\lambda$
be a parameter and $A$ an $n\times n$ matrix. Following Robbins and Rumsey \cite{RR}, we introduce the $\lambda$-determinant as follows.
\begin{align}\nonumber
\mathrm{det}_\lambda(A)&=\sum_{X\in \mathrm{Alt}_n}(-\lambda)^{i(X)}(1-\lambda^{-1})^{n(X)}A^X\\ \label{D:lambdadet}
&=\sum_{\pi\in\mathfrak S_n}(-\lambda)^{i(\pi)}a_{1\pi_1}\cdots a_{n\pi_n}+\sum_{X\in \mathrm{Alt}_n\backslash P_n}(-\lambda)^{i(X)}(1-\lambda^{-1})^{n(X)}A^X.
\end{align}

Note that the second equality holds even for $\lambda=1$ due to 
$0^0=1$.
We remark that the original $\lambda$-determinant \cite{RR} corresponds to a sign change of ours, while
our definition deforms the usual determinant at $\lambda=1$.
In fact when $\lambda=1$,
the second summand vanishes, therefore $\mathrm{det}_1(A)$ reduces to the usual determinant. In general, $\det_{\lambda}(A)$ is a rational
function in the variable $a_{ij}$ or rather a polynomial function in the variable $a_{ij}^{\pm 1}$. 
For any $3\times 3$ matrix $A=(a_{ij})$ 
\begin{align*}
{\det}_{\lambda}(A)&=a_{11}a_{22}a_{33}-\lambda a_{12}a_{21}a_{33}-\lambda^3 a_{13}a_{22}a_{31}-\lambda a_{11}a_{23}a_{32}\\
&+\lambda^2 a_{12}a_{23}a_{31}+\lambda^2 a_{13}a_{21}a_{32}+\lambda(1-\lambda) \frac{a_{12}a_{21}a_{23}a_{32}}{a_{22}},
\end{align*}
where the last term is due to the alternating sign matrix $Q$ in \eqref{e:3perm}.

Another example is a $\lambda$-deformation of
the Vandermonde determinant \cite{RR}: For commutating variables $x_i$, $1\leq i\leq n$, one has that
\begin{align}\label{F:lambdavandermonde}
{\det}_{\lambda}(x_i^{j-1})=\prod_{1\leq i<j\leq n}(x_j-\lambda x_i),
\end{align}
which is a continuous deformation of the Vandermonde determinant at $\lambda=1$.
When $\lambda=-1$, this formula has an interesting connection with the Pfaffian $\Pf$:
\begin{align}
\frac{{\det}_{1}(x_i^{j-1})}{{\det}_{-1}(x_i^{j-1})}=\prod_{1\leq i<j\leq n}\frac{x_j-x_i}{x_j+x_i}=\Pf(\frac{x_j-x_i}{x_j+x_i}),
\end{align}
where the last equality is Schur's formula for Pfaffian \cite{Ma}.
Here the Pfaffian $\Pf(M)$ of an antisymmetric matrix $M$ is defined by $\Pf(M)=\sqrt{\det(M)}$.

A permutation matrix $P$ of size $n$ sends the row vector $(1, 2, \ldots, n)$  to the vector $(\pi(1), \pi(2), \ldots, \pi(n))$ of its
corresponding permutation $\pi\in \mathfrak S_n$  by right matrix multiplication, i.e.
\begin{equation}\label{e:perm}
(1, 2, \cdots, n)P=(\pi(1), \pi(2), \ldots, \pi(n)),
\end{equation}
where $\pi_i$ are distinct integers from $[1, n]$.
Similarly $P$ also sends the column vector by left matrix multiplication:
$$P(1, 2, \ldots, )^t=(\pi^{-1}(1), \pi^{-1}(2), \ldots, \pi^{-1}(n))^t.
$$

 Suppose $X=(x_{ij})$ is an alternating sign matrix of size $n$, we claim that
 $X$ sends the vector $(1, 2, \cdots, n)^t$ to the vector $(X(1), X(2), \ldots, X(n))^t$ by right 
  matrix multiplication and
 $1\leq X(i)\leq n$. So for each alternating sign matrix $X$, we define the generalized permutation
 $(X(1), X(2),\ldots, X(n))$ by
 \begin{align}\label{e:genperm}
 X\begin{bmatrix} 1 \\ 2 \\ \vdots\\ n\end{bmatrix}:=
 \begin{bmatrix} X(1) \\  X(2)\\ \vdots \\ X(n)\end{bmatrix},
 \end{align}
 where $X(i)$ may not be distinct (cf. \eqref{e:perm}). Note that $(X(1), X(2), \ldots, X(n))$ descends to
 the permutation $(\pi(1), \pi(2), \cdots, \pi(n))$ when $X$ is the permutation matrix corresponding to
 $\pi\in\mathfrak S_n$. For example,
 $(1,2,3)Q=(2,2,2)$ for $Q$ in \eqref{e:3perm}.
Suppose there is $x_{ij}=-1$, then there are odd number of
nonzero entries along the $i$th row and they are alternatively $\pm 1$ located at the $j_1, \cdots, j_{2k+1}$th columns
such that $1\leq j_1<j_2<\cdots<j_{2k+1}\leq n$ and $k\geq 1$. Therefore
\begin{align*}
X(i)&=\sum_{j=1}^njx_{ij}=j_1-j_2+j_3-j_4+\cdots-j_{2k}+j_{2k+1}\\
&=j_1+(j_3-j_2)+\cdots+(j_{2k+1}-j_{2k})\geq k+1
\end{align*}
and
\begin{align*}
\sum_{j=1}^njx_{ij}&=(j_1-j_2)+\cdots+(j_{2k-1}-j_{2k})+j_{2k+1}< 
j_{2k+1}\leq n.
\end{align*}
Hence one always has that $1\leq X(i)=\sum_{j=1}^njx_{ij}\leq n$ for any alternating sign matrix $X=(x_{ij})$.

For fixed real numbers $q$ and $a$, we use the conventional $q$-analogue of $a$:
\begin{equation}
(a; q)_n=(1-a)(1-aq)\cdots (1-aq^{n-1}).
\end{equation}

With this preparation, we now introduce the
 notion of $\lambda$-hyperdeterminants.
\begin{definition}\label{D:lambdadeterminant}
Let $M=(M(i_{1},i_{2},\dotsc,i_{2m}))_{1\leq i_1,i_2,\dotsc, i_{2m}\leq n}$ be a $2m$ dimensional
hypermatrix, the $\lambda$-hyperdeterminant 
of $M$ is defined by
\begin{align}\label{D:qhyperdeterminant}
&\mathrm{det}^{[2m]}_{\lambda}(M)=\frac{(1-\lambda)^n}{(\lambda;\lambda)_n}\sum_{X_{1},X_2,\dotsc,X_{2m}\in \mathrm{Alt}_n}
\phi_\lambda(X_1,X_2,\dotsc,X_{2m})\\ \nonumber
&\qquad\qquad\qquad\times\prod_{i=1}^n M(X_1(i), X_2(i), \dotsc, X_{2m}(i)),
\end{align}
where the generalized sign factor $\phi_\lambda(X_1,\dotsc,X_{2m})$ is the following expression
\begin{align}\label{phi}
\prod_{r=1}^{m}(-\lambda^{r-1})^{i(X_r)}(1-\lambda^{r-1})^{n(X_r)}(-\lambda^r)^{i(X_{m+r})}(1-\lambda^{r})^{n(X_{m+r})},
\end{align}
and the generalized permutation $(X(1), \ldots, X(n))$ for $X\in \mathrm{Alt}_n$ was defined in \eqref{e:genperm}.
\end{definition}
Note that the $\lambda$-hyperdeterminant is different from the $q$-hyperdeterminant defined on
quantum linear group \cite{JZ}.

We also remark that
alternating sign matrices with $n(X_1)>0$ do not contribute to the sum, the first summation variable $X_1$ is therefore always taken as a permutation matrix. Moreover, one can even fix the
first indices to be $1, 2, \ldots, n$ and drop
the quantum factor. 

We now explain how the $\lambda$-hyperdeterminant
deforms the Cayley hyperdeterminant.

\begin{proposition} For any $2m$-dimensional hypermatrix
$$A=(a_{i_1,\ldots, i_{2m}})_{1\leq i_1, \ldots, i_{2m}\leq n}$$
one has that
\begin{equation}
\lim_{\lambda \to 1}\mathrm{det}^{[2m]}_{\lambda}(A)=\mathrm{det}^{[2m]}(A).
\end{equation}
\end{proposition}
\begin{proof}
When $\lambda=1$ the summands in the right side of (\ref{D:qhyperdeterminant}) (omitting the factor) vanish if one of $X_i$'s has at least one negative entry. Thus only permutation matrices contribute to the sum, so we can let
$X_r=P_{\sigma_r}=(\delta_{j\sigma_r(i)})\in \mathfrak S_{n}$ for $r=1, \cdots, 2m$.
Subsequently
$X_r(i)=\sum_{j=1}^{n} jX_r(i, j)=\sum_{j=1}^{n}j\delta_{j\sigma_r(i)}=\sigma_r(i)$,
and we see that the right side of (\ref{D:qhyperdeterminant}) matches exactly
with that of (\ref{D:hyperdeterminant}).
\end{proof}

\section{Hyperdeterminant formula for Macdonald functions}
We now generalize Matsumoto's
hyperdeterminant formula for Jack polynomials \cite{Ma} as an application of the $\lambda$-hyperdeterminant.

Recall that the classical Jacobi-Trudi formula expresses the Schur function associated to
partition $\mu$ as a determinant of simpler Schur functions of row shapes:
\begin{align}
s_{\mu}(x)=\det(s_{\mu_i-i+j}(x)).
\end{align}

In \cite{Ma}, Matsumoto gave a simple formula for Jack functions of rectangular shapes using the hyperdeterminant (see also \cite{BBL}).

\begin{proposition}\cite{Ma} \label{T:Matsutomo}
Let $k,s,m$ be positive integers. The rectangular Jack functions $Q_{(k^s)}(m^{-1})$ can be expressed
compactly as follows.
\begin{align}\label{F:hyperdeterminant}
\frac{(sm)!}{(m!)^{s}}Q_{(k^s)}(m^{-1})=s!\mathrm{det} ^{[2m]}(Q_{k+\sum_{j=1}^m (i_{m+j}-i_j)})_{1\leq i_1,\dotsc,i_{2m}\leq s}.
\end{align}
\end{proposition}
As the Schur function $s_{\mu}$ is the specialization $Q_{\mu}(1)$ of the Jack function, one sees immediately that this formula specializes to the Jacobi-Trudi formula for rectangular Schur functions.

Macdonald symmetric functions $Q_{\mu}(q,t)$ \cite{M} are a family of orthogonal symmetric functions with two parameters $q,t$ and labeled by partition $\mu$. In this paper we consider the case that $t=q^m$, $m$
any positive integer. When $q$ approaches to $1$, $Q_{\mu}(q,q^m)$ specializes to the Jack symmetric function $Q_{\mu}(m^{-1})$ \cite{M} which has 
many applications (eg. \cite{CJ1}). When the partition $\mu=(k^s)$, the
symmetric function $Q_{\mu}(q, t)$ or $Q_{\mu}(m^{-1})$ is referred as the rectangular (shaped) Macdonald function.

We now state our main result.
\begin{theorem}\label{T:rect}
Let $k,s,m$ be positive integers. Up to a scalar factor, the rectangular Macdonald function $Q_{(k^s)}(q,q^m)$
is a $2m$-dimensional $\lambda$-hyperdeterminant:
\begin{align}\nonumber
&\frac{(q;q)_{sm}}{(q;q)_m^s}Q_{(k^s)}(q,q^m)\\ \label{F:qhyperdetRecMac}
=&\frac{(q;q)_s}{(1-q)^s}\mathrm{det}^{[2m]}_q(Q_{k+\sum_{r=1}^m(i_{m+r}-i_r)}(q;q^m))_{1\leq i_1,\dotsc,i_{2m}\leq s},
\end{align}
where $\det^{[2m]}_q$ is the $\lambda$-hyperdeterminant with $\lambda=q$.
\end{theorem}
\begin{remark}
 When $q=1$, one recovers the hyperdeterminant formula (\ref{F:hyperdeterminant}) for Jack functions of rectangular shapes. The formula also gives another expression of the general formula of Lassalle-Schlosser \cite{LS} (see also
 \cite{W}) in the
 special shape.
\end{remark}
To prepare its proof, we first recall the following result from \cite{C}.

\begin{proposition}\label{T:loweringFor4RecMac}
Let $\rho=(k^s)$ be an rectangular partition with $k,s>0$. Let $m$ be a positive integer. Then $\frac{(q;q)_{sm}}{(q;q)_m^{s}} Q_{\rho}(q,q^m)$ is the coefficient of $z_1^kz_2^k\dotsm z_s^k$ in the Laurent polynomial $FG$, where
\begin{align}
\label{F:qDysonMac}
 F&=\prod_{1\leq i<j\leq s}\Big(z_iz_j^{-1};q\Big)_m
\Big(qz_i^{-1}z_j;q\Big)_m\\
G&=\prod_{i=1}^s\left(\sum_{j=0}^\infty Q_j(q,q^m)z_i^{j}\right).
\end{align}
\end{proposition}

We can now prove Theorem~\ref{T:rect}.
\begin{proof}
Recall that the $q$-Dyson Laurent polynomial \eqref{F:qDysonMac} is a generalization of the Vandemonde polynomial.
Using \eqref{F:lambdavandermonde}, we can rewrite the $q$-Dyson Laurent polynomial in the $z_i$ using $\lambda$-determinants.
\begin{align*}
F&=\prod_{1\leq i<j\leq s}\Big(z_iz_j^{-1};q\Big)_m
   \Big(qz_jz_i^{-1};q\Big)_m\\
 &=\prod_{1\leq i<j\leq s}\prod_{r=1}^{m}(z_j-q^{r-1}z_i)(z_j^{-1}-q^rz_i^{-1})\\
 &=\prod_{r=1}^{m}\mathrm{det}_{q^{r-1}}(z_i^{j-1})\mathrm{det}_{q^{r}}(z_i^{1-j})\quad \text{(using \eqref{F:lambdavandermonde} for s by s matrices)}\\
&=\prod_{r=1}^{m}\left(\sum_{X\in \mathrm{Alt}_s}(-q^{r-1})^{i(X)}(1-q^{r-1})^{n(X)}\prod_{i,j}z_i^{(j-1)X(i,j)} \right)\\
&\qquad\times\left(\sum_{X\in \mathrm{Alt}_s}(-q^{r})^{i(X)}(1-q^{r})^{n(X)}\prod_{i,j}z_i^{(1-j)X(i,j)}\right)\\
&=\sum_{X_1,\dotsc,X_{2m}\in \mathrm{Alt}_s}\prod_{r=1}^{m}(-q^{r-1})^{i(X_r)}(-q^{r})^{i(X_{m+r})}(1-q^{r-1})^{n(X_{r})}(1-q^r)^{n(X_{m+r})}\\
&\qquad \times\prod_{i=1}^sz_i^{\sum_{j=1}^s\sum_{r=1}^m j(X_r(i,j)-X_{m+r}(i,j))}\\
&=\sum_{X_1,\dotsc,X_{2m}\in \mathrm{Alt}_s}\phi_q(X_1,\dotsm,X_{2m})\prod_{i=1}^sz_i^{\sum_{j=1}^s\sum_{r=1}^m j(X_r(i,j)-X_{m+r}(i,j))}\\
\end{align*}
where in the last equality we have used the fact that the sum of entries in each row of an alternating sign matrix is one.
By Proposition~\ref{T:loweringFor4RecMac}, $\frac{(q;q)_{sm}}{(q;q)_m^s}Q_{(k^s)}(q,q^m)$ is the coefficient of $z_1^kz_2^k\dotsm z_s^k$ in
\begin{align*}
FG&=\sum_{X_1,\dotsc,X_{2m}\in \mathrm{Alt}_s}\Huge[\phi_q(X_1,\dotsm,X_{2m})\prod_{i=1}^sz_i^{\sum_{j=1}^s\sum_{r=1}^m j(X_r(i,j)-X_{m+r}(i,j))}\\
&\qquad\qquad\qquad \prod_{i=1}^s\left(\sum_{j=0}^\infty Q_j(q,q^m)z_i^{j}\right)\Huge],
\end{align*}
and this coefficient is seen as
\begin{align*}
\sum_{X_1,\dotsc,X_{2m}\in \mathrm{Alt}_s}\phi_q(X_1,\dotsm,X_{2m})
\prod_{i=1}^s Q_{k-\sum_{j=1}^s\sum_{r=1}^m j(X_r(i,j)-X_{m+r}(i,j))}(q,q^m).
\end{align*}

We now show formula (\ref{F:qhyperdetRecMac}) holds. The entries of matrix $M$ in Definition~\ref{D:lambdadeterminant} is now given by $M(i_1,i_2,\dotsc,i_{2m})=Q_{k+\sum_{r=1}^m(i_{m+r}-i_r)}(q,q^m)$. Then we have that
\begin{align*}
&\prod_{i=1}^sM\left(X_1(i), X_2(i), \ldots, X_{2m}(i)\right)\\ 
=&\prod_{i=1}^sM\left(\sum_{j=1}^sjX_1(i,j),\dotsm,\sum_{j=1}^sjX_{2m}(i,j)\right)\\
=&\prod_{i=1}^sQ_{k+\sum_{r=1}^m\left(\sum_{j=1}^sjX_{m+r}(i,j)-\sum_{j=1}^sX_r(i,j)\right)}(q,q^m)\\
=&\prod_{i=1}^sQ_{k+\sum_{r=1}^m\sum_{j=1}^sj(X_{m+r}(i,j)-X_r(i,j))}(q,q^m),
\end{align*}
which then implies the formula (\ref{F:qhyperdetRecMac}).
\end{proof}

\medskip

\centerline{\bf Acknowledgments}
NJ gratefully acknowledges the
partial support of Simons Foundation grant no. 523868, Humbolt Foundation,
NSFC grant no. 11531004
and MPI for Mathematics in the Sciences in Leipzig during this work.

\bibliographystyle{amsalpha}

\end{document}